\documentclass[preprint,12pt]{elsarticle}

\usepackage{amssymb}
\usepackage{amsthm}
\usepackage{amsmath,amssymb,amsopn,amsfonts,mathrsfs,amsbsy,amscd}
\usepackage{xcolor}
\usepackage{rotating}

\newcommand{\prs}{\langle\;,\;\rangle}
\newcommand{\too}{\longrightarrow}
\newcommand{\om}{\omega}
\newcommand{\esp}{\quad\mbox{and}\quad}

\def\br{[\;,\;]}

\newcommand{\G}{{\mathfrak{g}}}

\newcommand{\ad}{{\mathrm{ad}}}

\newcommand{\B}{{\cal B}}

\newcommand{\di}{\displaystyle}

\newcommand{\na}{\nabla}

\newcommand{\al}{\alpha}
\newcommand{\be}{\beta}

\newcommand{\ga}{\gamma}

\newcommand{\e}{\epsilon}

\newcommand{\ch}{\check{e}}
\newcommand{\la}{\lambda}
\newcommand{\De}{\Delta}
\newcommand{\de}{\delta}

\font\bb=msbm10

\def\B{\hbox{\bb B}}
\def\R{\hbox{\bb R}}

\def\N{\hbox{\bb N}}

\def\C{\hbox{\bb C}}

\newtheorem{theo}{Theorem}[section]
\newtheorem{pr}{Proposition}[section]
\newtheorem{Le}{Lemma}[section]

\newtheorem{rem}{Remark}[section]

\newtheorem{problem}{Problem}

\begin{document}
\begin{frontmatter}




\title{Poisson algebras and symmetric Leibniz bialgebra structures on oscillator Lie algebras}

\author[label1]{H. Albuquerque}
\address[label1]{CMUC, Departamento de Matem\'atica, Universidade de Coimbra, Apartado 3008
EC Santa Cruz
3001  501 Coimbra, Portugal\\e-mail: lena@mat.uc.pt \hspace{0.2cm} mefb@mat.uc.pt }

\fntext[label1]{The first and second authors were supported by the Centre for Mathematics of the
University of Coimbra -- UID/MAT/00324/2019, funded by the Portuguese
Government through FCT/MEC and co-funded by the European Regional
Development Fund through the Partnership Agreement PT2020. Fifth author was supported by the PCI of the UCA ``Teor\'\i a de Lie y Teor\'\i a de Espacios de Banach'', by the PAI with project numbers FQM298, FQM7156 and by the project of the Spanish Ministerio de Educaci\'on y Ciencia  MTM2013-41208P.}

\author[label1]{E. Barreiro}

\author[label2]{S. Benayadi}
\address[label2]{Laboratoire de Math\'{e}matiques IECL UMR CNRS 7502,
Universit\'{e} de Lorraine, 3 rue Augustin Fresnel, BP 45112,
F-57073 Metz Cedex 03, France\\e-mail: said.benayadi@univ-lorraine.fr}

\author[label3]{M. Boucetta}
\address[label3]{Universit\'e Cadi-Ayyad, Facult\'e des Sciences et Techniques, BP 549 Marrakech, Morocco\\e-mail: m.boucetta@uca.ac.ma}

\author[label4]{J.M. S\'{a}nchez}
\address[label4]{Departamento de Matematicas, Universidad de C\'{a}diz, Puerto Real, C\'{a}diz, Spain \\e-mail: txema.sanchez@uca.es}


\begin{abstract}
	Oscillator Lie algebras are the only non commutative  solvable Lie algebras which carry a bi-invariant Lorentzian metric. In this paper,
we determine all the Poisson structures,  and in particular, all   symmetric Leibniz  algebra structures whose underlying Lie algebra is an oscillator Lie algebra.     We give also all  the symmetric Leibniz bialgebra structures whose underlying Lie bialgebra structure is a Lie bialgebra structure on an oscillator Lie algebra. We derive some geometric consequences on oscillator Lie groups.
\end{abstract}

\end{frontmatter}







\section{Introduction}

A {\it Poisson algebra} is a finite dimensional Lie algebra $(\G,[\;,\;])$
endowed
with a commutative and associative product $\circ$ such that, for any $u,v,w\in\G$,
\begin{equation}
\label{eq0}[u,v\circ w]=[u,v]\circ w+v\circ [u,w].
\end{equation}
An algebra $(\mathrm{A},.)$ is called {\it Poisson admissible } if
$(\mathrm{A},[\;,\;],\circ)$ is a  Poisson algebra, where
\begin{equation}\label{oh}
[u,v]=u.v-v.u\esp u\circ v=\frac12(u.v+v.u).\end{equation}
Poisson algebras constitute an interesting topic in algebra and were studied by many
authors (see for instance \cite{bb, bai, GozRem,  MarRem, Shestakov}).

A {\it symmetric Leibniz algebra} is an algebra $(\mathrm{A},.)$ such that for any
$u\in\mathrm{A}$, the left multiplication $\mathrm{L}_u$ (i.e., $\mathrm{L}_u(v)=[u,v]$ for $v\in\mathrm{A}$) and the right multiplication $\mathrm{R}_u$ (i.e., $\mathrm{R}_u(v)=[v,u]$ for $v\in\mathrm{A}$) are derivations,
meaning that, for any $v,w\in\mathrm{A}$,
$$ u.(v.w)=(u.v).w+v.(u.w)\esp  (v.w).u=(v.u).w+v.(w.u).$$
It is well-known (see \cite{bb}) that if $(A,.)$ is a symmetric Leibniz algebra then it is Poisson admissible. Any Lie algebra is obviously a symmetric Leibniz algebra, however, the class of symmetric Leibniz algebras contains strictly the class of Lie
algebras. In \cite[Proposition 2.11]{be}, a useful characterization of symmetric Leibniz algebras is given. Note that symmetric Leibniz algebras constitute a subclass of left (right) Leibniz algebras introduced by Bloh \cite{bloh, bloh1} under the name of D-algebras and  rediscovered by  Loday \cite{loday}.

As a Lie bialgebra is a Lie algebra $\G$ and a Lie bracket on its dual $\G^*$ which are compatible in some sens (see Section \ref{section2}), a symmetric Leibniz bialgebra is a symmetric Leibniz algebra $A$ with a symmetric Leibniz product on its dual $A^*$ which are compatible.  The notion of Lie bialgebras was
introduced and studied by Drinfeld in \cite{dr} and symmetric Leibniz bialgebras were characterized by Barreiro and  Benayadi in 2016 (see \cite{be}). Note that any symmetric Leibniz bialgebra has an underlying structure of Lie bialgebra.

In this paper, we deal with the following problem which is natural by virtue of what above:
\begin{problem} Given a Lie algebra $\G$ (resp. a Lie bialgebra), determine all the Poisson structures on $\G$ (resp. symmetric Leibniz bialgebras) whose underlying Lie algebra (resp. Lie bialgebra) is $\G$.
	
\end{problem}
This problem has been solved for semi-simple Lie algebras since it was shown in \cite{bb, GozRem} that there is no non trivial Poisson structure on semi-simple Lie algebras and the purpose of this paper is to solve this problem for the class
of oscillator Lie algebras which as the semi-simple Lie algebras belong to the class of quadratic Lie algebras. Recall that a quadratic Lie algebra is a Lie algebra $(\G,\br)$ with a nondegenerate bilinear symmetric form $\prs$ which is bi-invariant, i.e.,
\[ \langle[u,v],w\rangle+\langle[u,w],v\rangle=0,\quad u,v,w\in\G. \]
The oscillator Lie algebras are the only non commutative  solvable Lie algebras which carry a bi-invariant Lorentzian metric \cite{m}. They are the Lie algebras of oscillator Lie groups. The oscillator group named so by Streater in \cite{streater}, as a four-dimensional connected, simply connected
Lie group, whose Lie algebra (know as the oscillator algebra) coincides with the one generated
by the differential operators, acting on functions of one variable, associated to the harmonic
oscillator problem. The oscillator group has been generalized in any even dimension $2n \geq 4$,
and several aspects of its geometry have been intensively studied, both in differential geometry
and in mathematical physics. Here are  few examples of studies on oscillator groups:  Lie bialgebras structures and Yang-Baxter equation \cite{bm},  Einstein-Yang-Mills equations \cite{duran}, parallel hypersurfaces \cite{calvaruso}, homogeneous structures \cite{gadea}, electromagnetic waves
\cite{levi}, the Laplace-Beltrami operator \cite{muller}.

For $n \in \N^*$ and $\la = (\la_1,\ldots,\la_n) \in \R^n$ with $0 < \la_1 \leq \cdots \leq \la_n$, the $\la$-oscillator group, denoted by $G_\la$, is the Lie group with the underlying manifold $\R^{2n+2} = \R \times \R\times \C^n$ and product $$(t,s,z).(t',s',z') = \Bigl( t+t',s+s'+\frac12 \sum_{j=1}^n \mbox{Im }[\bar{z}_j \exp(it\la_j)z_j'],
\ldots, z_j+\exp(it\la_j)z'_j, \ldots \Bigr) .$$
It is Lie algebra is $\G_\la$ is $\R \times \R\times \C^n$ with its canonical basis
$\B = \left\{e_{-1},e_0,e_j,\ch_j\right\}_{j=1,\ldots,n}$ with
\[ e_{-1}=(1,0,0),e_0=(0,1,0),\; e_j=(0,0,(0,\ldots,1,\ldots,0))\esp \ch_j=
(0,0,(0,\ldots,\imath,\ldots,0)). \]
 and the Lie brackets  are given by
\begin{equation}\label{bracket}
[e_{-1},e_i] = \la_i \ch_i, \qquad[e_{-1},\ch_i] = -\la_ie_i, \qquad[e_i,\ch_i] = e_0,
\end{equation}
for $i=1,\ldots,n$, the unspecified products are either given by antisymmetry or  zero.  For $x\in\G_\la$, let
$$x=x_{-1}e_{-1}+x_0e_0+\sum_{i=1}^n(x_ie_i+\check{x}_i\check{e}_i).$$
The nondegenerate symmetric bilinear form
\begin{equation}\label{lorentz}\textbf{k}_\la(x,x):=2x_{-1}x_0+\sum_{j=1}^n\frac1{\la_j}(x_j^2+\check{x}_j^2)\end{equation}satisfies
$$\textbf{k}_\la([x,y],z)+\textbf{k}_\la(y,[x,z])=0,\quad\mbox{for
	any}\quad x,y,z\in\G_\la$$and hence defines  a
Lorentzian bi-invariant metric on $\G_\la$.

The main results of this paper are:
\begin{enumerate}
	\item The determination of all the Poisson structures whose underlying Lie algebra is an oscillator Lie algebra (see Theorem \ref{poisson}). We show that contrary to the semi-simple case, there is a family depending on a one parameter of non trivial Poisson structures on any given oscillator Lie algebra. Moreover, these structure are symmetric Leibniz structures and are invariant with respect to $\textbf{k}_\la$.
	\item The determination of all symmetric Leibniz bialgebras structures whose underlying Lie bialgebra is a bialgebra structure on an oscillator Lie algebra (see Theorem \ref{bel}).
	\item The use of the   algebraic results found in 1. and 2. to derive some geometric consequences at the level of the oscillator Lie groups. Indeed, it has been shown in \cite{bb} that if $G$ is a Lie group and $(\G,\br)$ its Lie algebra then there is  a bijection
	between the set of Poisson structures on   $(\G,\br)$ and the space of
	bi-invariant torsion free linear connections  on $G$ which have the same curvature
	as $\na^0$ where $\na^0$ is the bi-invariant connection on $G$ given by $\na^0_XY=\frac12[X,Y]$. Moreover, if the Poisson structure is symmetric Leibniz the corresponding bi-invariant connection is locally symmetric and has the same holonomy algebra as $\na^0$. In the case of an oscillator Lie algebra $\G_\la$, the connection $\na^0$ is the Levi-Civita connection of the left invariant Lorentzian metric associated to $\textbf{k}_\la$. So any non trivial Poisson structure on $\G_\la$ determine a bi-invariant connection on $G_\la$ which has the same geometric properties of $\na^0$ except one, namely, it does not preserves the metric (see Section \ref{section5}).
\end{enumerate}

The paper is organized as follows. In Section \ref{section2}, we give some known results needed in the paper. We give also some new results on oscillator Lie algebras, namely, some type of its derivations (see Lemma \ref{derexp}) and its quadratic dimension (see Lemma \ref{le1}).  Sections \ref{section3}-\ref{section5} are devoted to the results enumerated above.

{\bf Notations and conventions.} Let $V$ be a finite dimensional vector space. We identify $V\otimes V$ to the space of bilinear forms on $V^*$,  $V\wedge V$  and $V\odot V$, respectively,  the subspaces of skew-symmetric (resp. symmetric) bilinear forms on $V^*$. For any $u,v\in V$, $u\otimes v$ is the element of $V\otimes V$ given by
\[ (u\otimes v)(\al,\be)=\al(u)\be(v),\quad \al,\be\in V^*. \]
We put also
\[ u\wedge v=u\otimes v-v\otimes u\esp u\odot v=u\otimes v+v\otimes u.  \]
For any $J\in\mathrm{End}(V)$ and $r\in V\otimes V$ we denote by $J(r)$ or $Jr$ the element of $V\otimes V$ given by
\[ J(r)(\al,\be)=r(J^*\al,\be)+r(\al, J^*\be),\quad \al,\be\in V^*. \]
Note that $J(u\otimes v)=J(u)\otimes v+u\otimes J(v)$.


\section{Preliminaries}\label{section2}

In this section, we recall the definitions of  Lie bialgebras and symmetric Leibniz bialgebras and their characterization given in \cite{be}. We give the characterization of the Lie bialgebras structures on oscillator Lie algebras given in \cite{bm}. We prove two lemmas interesting in their owns and which will be useful later.

\subsection{Lie bialgebras and symmetric Leibniz bialgebras}  Let $(\mathcal{A},\centerdot)$ be an algebra and denote by   ${\mathrm{L}}_u$ and ${\mathrm{R}}_u$, respectively, the left and the right multiplication by $u$. We define ${\mathrm{L}}^*,{\mathrm{R}}^*:\mathcal{A}\too\mathrm{End}(\mathcal{A}^*)$ by
\[ \prec\mathrm{L}^*_x\al,y\succ=\prec\al,{\mathrm{R}}_xy \succ
\esp \prec\mathrm{R}^*_x\al,y\succ=\prec\al,{\mathrm{L}}_xy \succ.\]

Suppose that we have also an algebra product $\circ$ on the dual $\mathcal{A}^*$ defined by a coproduct
$\De:\mathcal{A}\too \mathcal{A}\otimes \mathcal{A}$.  Define on $\Phi(\mathcal{A})=\mathcal{A}\oplus \mathcal{A}^*$ the product $'.'$ by
\begin{equation} \label{eq1}(x+\al).(y+\be)=x\centerdot y+{\mathrm{L}}_\al^* y+{\mathrm{R}}_\be^*x+\al\circ\be +{\mathrm{L}}_x^*\be+{\mathrm{R}}_y^*\al,\; x,y\in \mathcal{A},\al,\be\in \mathcal{A}^*, \end{equation}where $x,y\in \mathcal{A}$ and $\al,\be\in \mathcal{A}^*$.
Thus $(\Phi(\mathcal{A}),.)$ is an algebra and both $(\mathcal{A},\centerdot)$ and $(\mathcal{A}^*,\circ)$ are subalgebras of $\Phi(\mathcal{A})$. Moreover, the product $'.'$ satisfies
\[ \mathcal{B}(u.v,w)=\mathcal{B}(u,v.w),\; u,v,w\in \Phi(\mathcal{A}), \]where $\mathcal{B}(x+\al,y+\be)=\prec \al,y\succ+\prec\be,x\succ$.

 Suppose that $(\mathcal{A},\centerdot)$ is a Lie algebra (resp. a symmetric Leibniz algebra). The triple $(\mathcal{A},\centerdot,\De)$ is a called a Lie bialgebra  (resp. a symmetric Leibniz bialgebra) if $(\Phi(\mathcal{A}),.)$ is a Lie algebra (resp. a symmetric Leibniz algebra).

The following well-known result characterizes Lie bialgebras.
\begin{pr} Let $(\G,\br)$ be a Lie algebra and $\De:\G\too\otimes^2\G$ a coproduct. Then $(\G,\br,\De)$ is a Lie bialgebra if and only if $\De^*:\G^*\otimes\G^*\too\G^*$ is a Lie bracket on $\G^*$ and $\De$ is 2-cocycle with respect to the adjoint action, i.e.,
	\[ \De([u,v])=\ad_u\De(v)-\ad_v\De(u)\quad u,v\in\G. \]
	
\end{pr}

The characterization of symmetric Leibniz bialgebras was done in \cite{be} and it is more subtle. We present now this characterization.

 Let $({\frak L},.)$ be a symmetric Leibniz algebra equipped with a coproduct $\De$. We put $$[x,y] := \frac12(x.y - y.x), \hspace{0.4cm} x \bullet y := \frac12(x.y+ y.x),$$ for any $x,y \in {\frak L}.$  Then $({\frak L},[\;,\;])$ is a Lie algebra and $({\frak L},\bullet)$ is a commutative associative algebra. For $x,y,z \in \frak L$ we define $$x\cdot_a(y \otimes z) := (x \bullet y) \otimes z, \hspace{0.4cm} (y \otimes z)\cdot_ax := y \otimes(z \bullet x),$$ $$x\cdot_{\frak L}(y \otimes z) := ([x,y]) \otimes z, \hspace{0.4cm} (y \otimes z)\cdot_{\frak L}x := y \otimes([z, x]).$$
We also define two new coproducts on ${\frak L}$: $$\De_{\frak L} := \frac12(\De - \tau \circ \De), \hspace{0.4cm} \De_a := \frac12(\De + \tau \circ \De),$$ where, we use the twist map $\tau(x \otimes y) = y \otimes x$ for $x,y \in {\frak L}.$
With these notations  in mind, we have the following useful characterization of symmetric Leibniz bialgebras.
\begin{theo}{\rm\cite[Theorem 3.8]{be}}\label{be} Let $(\mathfrak{L},.)$ be a symmetric Leibniz algebra and $\De$ a coproduct on $\mathfrak{L}$. The triple $(\mathfrak{L},.,\De)$ is a symmetric Leibniz bialgebra if and only if
	\begin{enumerate}
		\item $(\mathfrak{L},[\;,\;],\De_{\mathfrak{L}})$ is a Lie bialgebra,
		\item $(\De_a\otimes I_{\frak L}) \circ \De_a = (\De_a \otimes I_{\frak L}) \circ \De_{\frak L} = (\De_{\frak L} \otimes I_{\frak L}) \circ \De_a = 0,$
		\item $\De_a(x \bullet y) = \De_{\frak L}(x \bullet y) = \De_a([x,y]) = 0,$
		\item $y\cdot_a\De_a(x) + \De_a(y)\cdot_ax = 0,$
		\item $x\cdot_a\De_{\frak L}(y) - \De_a(x)\cdot_{\frak L}y = 0,$
		\item $x\cdot_{\frak L}\De_a(y) - \De_{\frak L}(x)\cdot_ay = 0,$
	\end{enumerate}
	for any $x, y \in \frak L.$
\end{theo}

\begin{rem}\label{rem:be} In the original publication of the paper \cite{be} the statement of  Theorem 3.8 has a misprint. A corrected version of Theorem 3.8 in \cite{be} is provided in the present paper, where the Condition 1. of the Theorem is rephrased as  ``$(\mathfrak{L},[\;,\;],\De_{\mathfrak{L}})$ is a Lie bialgebra''.
\end{rem}

\subsection{Lie bialgebra structures on the oscillator Lie algebras}

For $n \in \N^*$ and $\la = (\la_1,\ldots,\la_n) \in \R^n$ with $0 < \la_1 \leq \cdots \leq \la_n$, let
 $\G_\la$ be the oscillator Lie algebra defined in \eqref{bracket}.
We denote by $S$ the $\R$-vector subspace of $\G_\la$ spanned by $\{e_i,\ch_i\}_{i=1,\ldots,n}$ and by $\om$ the 2-form on $\G_\la$ given by $$i_{e_{-1}} \om = i_{e_0}\om = 0,\; \om(e_i,e_j) = \om(\ch_i,\ch_j) = 0 \quad \mbox{and} \quad \om(e_i,\ch_j) = - \om(\ch_i,e_j) = \de_{ij},$$ for any $1 \leq i, j \leq n.$ The restriction of $\om$ to $S$ is a symplectic 2-form (that is,  $\om$ is a  skew-symmetry nondegenerate bilinear form of $S$) and, for any $u, v \in S$,
\begin{equation}\label{bracket1}
[u,v] = \om(u,v)e_0.
\end{equation}
  We denote by $D$ the restriction of $\ad_{e_{-1}}$ to $S,$ that is, $D(u) = [e_{-1},u]$ for all $u \in S.$ Clearly, $D$ is an isomorphism which is skew-symmetric with respect $\om$, meaning that
\begin{equation}\label{ad}
\om(D(u),v) + \om(u,D(v)) = 0,
\end{equation}
for any $u,v \in S.$

We call  $\G_\la$, \emph{generic} if the scalars satisfy $0 < \la_1 <\cdots < \la_n$ and $\la_k \neq \la_i + \la_j$, for any $1 \leq i < j < k \leq n.$

The following lemma appeared first in \cite{bm} without proof, we give its proof here.

\begin{Le}\label{derexp} If $0<\la_1<\ldots<\la_n$ and  $J :\G_\la\too\G_\la$ is a derivation satisfying  $J(e_0)=J(e_{-1})=0$ then there exists $(\mu_1,\ldots,\mu_n) \in \R^n$ such that,
	$$J (e_i)=  \mu_{i}  \ch_i \mbox{ and }J (\ch_i)=- \mu_{i} e_i ,$$
	for any $i \in \{1, \ldots,n\}$.
\end{Le}

\begin{proof} Since $J(e_0)=J(e_{-1})=0$,  $J$ is a derivation if and only if
	\[ J\circ\ad_{e_{-1}}=\ad_{e_{-1}}\circ J\esp \om(Ju,v)+\om(u,Jv)=0, \]
	Since $\ad_{e_{-1}}$ leaves invariant $S$ and its restriction to $S$ is bijective, the first relation implies that $J(S)\subset S$. The eigenvalues of the restriction of $\ad_{e_{-1}}^2$ to $S$ are $-\la_1^2>\ldots>-\la_n^2$ and the corresponding eigenspaces are $E_i=\mathrm{span}\{e_i,\ch_i\}$, $i=1,\ldots,n$.
	Now $J$ commutes with $\ad_{e_{-1}}^2$ and hence leaves invariant the $E_i$ for $i=1\ldots,n$. So in restriction to $E_i$ the matrices of $J$ and $\ad_{e_{-1}}$ are given, respectively, by
	\[ \left( \begin{array}{cc}a&b\\c&d\\
	\end{array} \right)\esp \left( \begin{array}{cc}0&-\la_i\\\la_i&0\\
	\end{array} \right). \] These matrices commute if and only if $a=d$ and $b=-c$. Moreover, the relation $\om(Je_i,\ch_i)+\om(e_i,J\ch_i)=0$ gives that $a=0$ which completes the proof.
\end{proof}

For any $\mu\in\R^n$ we will denote by $J^\mu$ the endomorphism of $\G_\la$ satisfying
\[ J^\mu(e_0)=J^\mu(e_{-1})=0, J^\mu(e_i)=\mu_i\ch_i\esp J^\mu(\ch_i)=-\mu_ie_i,\quad i=1,\ldots,n.  \]



For any $r \in\wedge^2\G_\la$ consider the endomorphism $r_{\#}:\G_\la^*\too\G_\la$ such that
$\be(r_{\#}(\al))=r(\al,\be)$, for any $\alpha, \beta  \in \G_\la^*$.
For any $r_1,r_2\in\wedge^2\G_\la$, let $\om_{r_1,r_2}$ be the element of $\wedge^2\G_\la$ defined by
$$\om_{r_1,r_2}(\al,\be)=\frac12\left(\om(r_{1\#}(\al),r_{2\#}(\be))+\om(r_{2\#}(\al),r_{1\#}(\be)\right),$$
for any $\alpha, \beta  \in \G_\la^*$. As usual, $\B^* = \left\{e_{-1}^*,e_0^*,e_i^*,\ch_i^*\right\}_{i=1,\ldots,n}$ is the dual basis of $\B$ and we denote  by $\wedge^2S$ the subspace of $\wedge^2\G_\la$ consisting of $r$ satisfying $r_{\#}(e_{-1}^*)=r_{\#}(e_{0}^*)=0$ and $S^*=\mathrm{span}\{e_i^*,\ch_i^*\}_{i=1,\ldots,n}$.

The Lie bialgebra structures on  generic oscillator Lie algebras  were given in  \cite{bm}. We state Theorem 1.1 in \cite{bm} with a slight improvement thank to Lemma \ref{derexp}.

\begin{theo}{\rm\cite[Theorem 1.1]{bm}} \label{bm} Let $\G_\la$ be a generic oscillator Lie algebra and  $\De :\G_\la\too \otimes^2 \G_\la$ a coproduct. Then $(\G_\la,\br,\De)$ is a Lie bialgebra  if and only if there exists $r\in\wedge^2S$, $u_0\in S$ and $\mu\in\R^n$ such that
\begin{eqnarray*}
&& \De(u)= \ad_ur+2e_0\wedge(( J^\mu+ \ad_{u_{0}} )(u)),
\end{eqnarray*}
and
\begin{equation}\label{bou}
\om_{r,\ad_{e_{-1}} r}-(J^\mu\circ \ad_{e_{-1}})  r=0.
\end{equation}
Moreover, in this case, the Lie bracket $\br^*=\De^*$ on $\G_\la^*$  satisfies
\begin{equation}\label{bracketmain1}\left\{\begin{array}{l}
\;[e_{0}^*,\al]^\ast =2(J^\mu)^*\al-2(\ad^*_{e_{-1}}\al)(u_0)e_{-1}^*+ i_{r_{\#}(\al)}\om,\\
\;[\al,\be]^\ast =\ad_{e_{-1}} r(\al,\be)e_{-1}^*,
\end{array}\right.
\end{equation}
for $\alpha, \beta \in S^*$ and $e_{-1}^*$ is a central element.
\end{theo}

\begin{rem}\label{rem2} Note that \eqref{bou} involves only data on the vector space $S$ and if $\dim\G_\la=4$ then $r=a e_i\wedge\ch_i$ and  $\ad_{e_{-1}}r=0$ hence the equation \eqref{bou} is satisfied.
	
	\end{rem}

We end this section by giving the quadratic dimension of the oscillator Lie algebras. This notion was introduced by Benayadi in \cite{ben, bajo} and turn out to provide much relevant information on the structure of the Lie algebra.
Let $\G$ be a Lie algebra and let $B(\G)$ be the vector space spanned by all non-degenerate invariant symmetric bilinear forms on $\G$. Let us denote $dq (\G) := \dim B(\G)$; we say that $dq (\G)$ is the {\it quadratic dimension} of $\G$. The following lemma gives the quadratic dimension of the oscillator Lie algebras and will play a crucial role in the proof of Theorem \ref{poisson}.

\begin{Le}\label{le1} The quadratic dimension of $\G_\la$ is equal to 2. More precisely,
	let  $B:\G_\la\times \G_\la\too\R$ be a symmetric bilinear invariant form, i.e.,
	\begin{equation}\label{lep}
	B([u,v],w)+B(v,[u,w])=0,\quad \mbox{ for } u,v,w\in\G_\la.
	\end{equation}Then there exist $p,q\in\R$ such that
	\[ B=p \textbf{k}_\la+qe_{-1}^*\odot e_{-1}^*, \]where $\textbf{k}_\la$ is the invariant Lorentzian metric given by \eqref{lorentz}.
	
\end{Le}
\begin{proof} Since $\textbf{k}_\la$ is non-degenerate there exists a $\textbf{k}_\la$-symmetric endomorphism $L$ of $\G_\al$ such that for any $u,v\in\G_\la$, $B(u,v)=\textbf{k}_\la(Lu,v)$ and $B$ is bi-invariant if and only if
	\[ L[u,v]=[u,Lv];\quad u,v\in\G. \]
	This condition implies that the center of $\G_\la$ is invariant by $L$ and hence there exits $p\in\R$ such that $L(e_0)=p e_0$. Moreover, $L\circ\ad_{e_{-1}}=\ad_{e_{-1}}\circ L$ and hence $L$ leaves $S$ invariant. So we get for any $u,v\in S$,
	\[ \om(u,v)pe_0=\om(u,Lv)e_0. \]But the restriction of $\om$ to $S$ is non-degenerate and hence, for any $u\in S$, $Lu=pu$. On the other hand, $L$ is symmetric and hence leaves invariant $S^\perp=\mathrm{span}\{e_0,e_{-1} \}$ so there exits $q'$ such that $L(e_{-1})=q'e_0+pe_{-1}$. By taking $q'=1/2 q$ we get the result.
\end{proof}

\section{Poisson  and symmetric Leibniz structures on oscillator algebras}\label{section3}

In this section we will determine  all Poisson structures, and in particular, all   Leibniz symmetric algebra structures whose underlying Lie algebra are generic  oscillator Lie algebras.

\begin{theo}\label{poisson}Let $(\G_\la,[\;,\;])$ be an oscillator Lie algebra with $0<\la_1<\ldots<\la_n$ and  $\circ:\G_\la\otimes\G_\la\too\G_\la$ a product such that $(\G_\la,\br,\circ)$ is a Poisson algebra. Then there exists $c\in\R$ such
	$$ e_{-1}\circ e_{-1}=ce_0$$ and all the other  products are zero.
Moreover,  $(\G_\la,.)$ is a symmetric Leibniz algebra where
\[ u.v=[u,v]+u\circ v,\quad u,v\in\G_\la \]and $\textbf{k}_\la$ is invariant with respect to the product $'.'$, i.e., for any $u,v,w\in\G_\la$,
\[ \textbf{k}_\la(u.v,w)=\textbf{k}_\la(u,v.w). \]

\end{theo}

\begin{proof}  Consider a Poisson structure on the oscillator Lie algebra $(\G_\la,[\;,\;])$ given by an associative and commutative product $\circ$ satisfying
	\begin{equation}
	\label{eq3}[u,v\circ w]=[u,v]\circ w+v\circ [u,w].
	\end{equation}
For $ u,v \in \G_\la$, put
\[ u\circ v=a_{-1}(u,v)e_{-1}+a_{0}(u,v)e_{0}+a_S(u,v),\]
with $a_{-1} $ and $a_0 $ two symmetric bilinear  forms and $a_S:\G_\la\times \G_\la\too S$ a bilinear symmetric  map.

It is obvious from \eqref{eq3} that the center is stable by the product and hence  $e_0\bullet e_0=ce_0$, with $c:=a_{0}(e_0,e_0 )$.

Now for $ u,v,w \in \G_\la$, \eqref{eq3} can be written
\begin{equation*}
\begin{split}
a_{-1}(v,w)[u,e_{-1}]+[u,a_S(v,w)]  = &   \Bigl( a_{-1}([u,v],w)+a_{-1}(v,[u,w])\Bigr) e_{-1} \\
 & \hspace{-4,5cm}+ \Bigl( a_{0}([u,v],w)+a_{0}(v,[u,w])\Bigr) e_{0}   +a_S([u,v],w)+a_S(v,[u,w]).
\end{split}
\end{equation*}Since $e_{-1} \notin [\G_\la , \G_\la ]$, this is equivalent to
\begin{equation}\label{p1}
\left\{  \begin{array}{l}a_{-1}([u,v],w)+a_{-1}(v,[u,w])=0\\
	a_{-1}(v,w)[u,e_{-1}]+[u,a_S(v,w)]  =  \Bigl( a_{0}([u,v],w)+a_{0}(v,[u,w])\Bigr) e_{0} \\
	 +a_S([u,v],w)+a_S(v,[u,w]).\end{array}
	 \right.
\end{equation}
By Lemma \ref{le1}, it follows that
\begin{equation}\label{-1}a_{-1}=p \textbf{k}_\la+qe_{-1}^*\odot e_{-1}^*.\end{equation}
If we take $v=e_0$ and $u,w \in S$ in  \eqref{p1}, we get
\begin{equation*}
\begin{split}
a_{-1}(e_0,w)[u,e_{-1}]+[u,a_S(e_0,w)]  = &  a_{0}(e_0,[u,w] ) e_{0} +a_S(e_0,[u,w]).
\end{split}
\end{equation*}We have $[u,a_S(e_0,w)] =\om(u,a_S(e_0,w)) e_{0} $ and, by virtue of \eqref{-1}, $a_{-1}(e_0,w)=0$ so
\[ \om(u,a_S(e_0,w))=c\om(u,w)\esp \om(u,w)a_S(e_0,e_0)=0. \]
Since $\om$ is nondegenerate in restriction to $S$ then  $a_S(e_0,w)=cw$ for all $w\in S$ and $a_S(e_0,e_0)=0$.

Using  \eqref{p1} with  $v=e_0$, $u=e_{-1}$ and $w\in S$ we get
\begin{equation*}
\begin{split}
[e_{-1},cw]  = &  a_{0}(e_0,[e_{-1},w])  e_{0} +a_S(e_0,[e_{-1},w]),
\end{split}
\end{equation*}
so
\[ c[e_{-1},w]=a_{S}(e_0,[e_{-1},w])\esp a_{0}(e_0,[e_{-1},w])=0. \]
Since the restriction of $\ad_{e_{-1}}$ to $S$ is bijective then by the last condition we conclude $a_0(e_0,S)=0$.

If we take  $v=e_0$, $w=e_{-1}$ and  $u\in S$ in  \eqref{p1}, we obtain
\begin{equation*}
\begin{split}
a_{-1}(e_0,e_{-1})[u,e_{-1}]+[u,a_S(e_0,e_{-1})]  = & a_{0}(e_0,[u,e_{-1}]) e_{0} +a_S(e_0,[u,e_{-1}])
\end{split}
\end{equation*}
Since $a_0(e_0,S)=0$, this equation implies that $[u,a_S(e_0,e_{-1})]=\om(u,a_S(e_0,e_{-1}))=0$ and $a_{-1}(e_0,e_{-1})=p=c$ and then
$a_S(e_0,e_{-1})=0$.

If we take $u,v,w\in S$ in  \eqref{p1} and use the facts that  $a_S(e_0,w)=cw$ and $a_0(e_0,S)=0$, we get
\begin{eqnarray*}
 a_{-1}(v,w)[u,e_{-1}] = c\Bigl(\om(u,v)w+\om(u,w)v\Bigr), \esp
  \om(u,a_S(v,w))  =    0  .\nonumber
\end{eqnarray*}
This implies that $a_S(v,w)=0$ and by taking
 $v=e_i$ and $w=\ch_i$ and by using the fact that $a_{-1}(e_i,\ch_i)=0$, we get
\[ c\om(u,e_i)=c\om(u,\ch_i)=0  \quad \mbox{ for } u\in S,\] and hence $c=0$. This implies that $a_{-1}=qe_{-1}^*\odot e_{-1}^*$.

If we take $u=e_{-1}$ in  \eqref{p1} we get
\begin{eqnarray*}
	&& [e_{-1},a_S(v,w)]=a_S([e_{-1},v],w)+a_S(v,[e_{-1},w]), \\
	&& a_0([e_{-1},v],w)+a_0(v,[e_{-1},w])=0.
\end{eqnarray*}
If we take $v=w=e_{-1}$ in the first equation we get $ [e_{-1},a_S(e_{-1},e_{-1})]=0$, since $ a_S(e_{-1},e_{-1})  \in S$ then $ a_S(e_{-1},e_{-1}) =0$.
If we take $v=e_{-1}$ in the second equation we get $a_0(e_{-1},S)=0$.

 We have shown so far that, for all $u,v \in S$
$$\begin{array}{l}
\left\{
\begin{array}{l}
 e_0\circ e_0 = e_0\circ u=0,\\
u\circ v=a_0(u,v)e_0,\\
  e_0\circ e_{-1}=a_0(e_0,e_{-1})e_0,\\
  e_{-1}\circ u=a_S( e_{-1},u),\\
  e_{-1}\circ e_{-1}=qe_{-1}+a_0(e_{-1},e_{-1})e_0.
\end{array}
\right.
\end{array}$$

We define the endomorphisms $A,L:S\too S$ by the relation: for  $u,v\in S$,
\[ a_0(u,v)=\om(A(u),v)=\om(A(v),u)\esp L(u):= e_{-1}\bullet u=a_S(e_{-1},u). \]
It is clear that $A$ the skew-symmetric with respect to  $\om$.

If we take $v=e_{-1}$ and $u\in S$ in  \eqref{p1} then
\begin{eqnarray}
&& a_{-1}(e_{-1},w)[u,e_{-1}]=a_S([u,e_{-1}],w)+a_S(e_{-1},[u,w]), \label{eq:33a}\\
&&  \om(u,a_S(e_{-1},w))=a_{0}([u,e_{-1}],w)+a_{0}(e_{-1},[u,w]). \label{eq:33b}
\end{eqnarray}
If we take $w\in S$,  then  the last equation can be written
\[ \om(u,L(w))=-\om(AD(u),w)+\om(u,w)a_0(e_{-1},e_0), \]where $D$ is the restriction of $\ad_{e_{-1}}$ to $S$.
Since $\om$ is nondegenerate in restriction to $S$,  $A$ and $D$ are skew-symmetric with respect to $\om$,  then
\[ L=-DA+a_0(e_{-1},e_0)\mathrm{Id}_S. \]
If we take $w=e_{-1}$ and $u\in S$ in the equation \eqref{eq:33a}, we get
\[ LD=\frac12a_{-1}(e_{-1},e_{-1})D. \]
As $D$ is invertible then  $L=\frac12a_{-1}(e_{-1},e_{-1})\mathrm{Id}_S$ and
\begin{equation}\label{eq:DA}
DA= \Bigl(a_0(e_{-1},e_0)-\frac12a_{-1}(e_{-1},e_{-1}) \Bigr)\mathrm{Id}_S.
\end{equation}
Let us summarize what we have shown until now: for $u,v\in S,$
$$\begin{array}{l}
\left\{
\begin{array}{l}
 e_0\circ e_0 = e_0\circ u=0,\\
u\circ v=\om(A(u),v)e_0,\\
  e_0\circ e_{-1}=a_0(e_0,e_{-1})e_0,\\
  e_{-1}\circ e_{-1}=a_{-1}(e_{-1},e_{-1})e_{-1}+a_0(e_{-1},e_{-1})e_0,\\
  e_{-1}\circ u=\frac12a_{-1}(e_{-1},e_{-1})u .
\end{array}
\right.
\end{array}$$

We need now to use  the associativity of the product $\circ$ in order to finish the proof. We have, for any $u,v,w\in\G_\la$
\[ \mathrm{ass}(u,v,w):=(u\circ v)\circ w-u\circ(v\circ w)=0. \]
So, for any $u,v\in S$,
\begin{eqnarray*}
0&=&\mathrm{ass}(e_{-1},e_{-1},e_0)  =  a_{-1}(e_{-1},e_{-1})e_{-1}\circ e_0-a_0(e_0,e_{-1})e_{-1}\circ e_0 \\
 & =&\Bigl(a_{-1}(e_{-1},e_{-1})-a_0(e_0,e_{-1})\Bigr)a_0(e_0,e_{-1})e_0,\\
0&=&\mathrm{ass}(e_{-1},u,v)  =  \frac12a_{-1}(e_{-1},e_{-1})u\circ v-\om(A(u),v)e_{-1}\circ e_0 \\
 & =&\Bigl(\frac12a_{-1}(e_{-1},e_{-1}) - a_0(e_0,e_{-1})\Bigr)\om(A(u),v) e_0.
\end{eqnarray*}So, for any $u\in S$,
\[\Bigl(a_{-1}(e_{-1},e_{-1})-a_0(e_0,e_{-1})\Bigr)a_0(e_0,e_{-1})=0\esp \Bigl(\frac12a_{-1}(e_{-1},e_{-1}) - a_0(e_0,e_{-1})\Bigr)Au=0. \]
 If there exists $u\in S$ such that $Au\neq 0$ then $ \frac12a_{-1}(e_{-1},e_{-1})-a_0(e_0,e_{-1})=0$. By the equation \eqref{eq:DA} we have $DA=0$ and as $D$ is invertible we conclude $A=0$, a contradiction. We have $A=0$ then $DA=0$ and consequently by equation \eqref{eq:DA}, we get   $a_{-1}(e_{-1},e_{-1})=2a_0(e_0,e_{-1})$. Now, using also $\mathrm{ass}(e_{-1},e_{-1},e_0)=0$, we have $a_{-1}(e_{-1},e_{-1})=2a_0(e_0,e_{-1})$, $A=0$ then $a_0(e_0,e_{-1})^2=0$, therefore $a_0(e_0,e_{-1})=0$, $a_{-1}(e_{-1},e_{-1})=0$ and $A=0$.
To summarize, we have shown that the only non vanishing product is
\[ e_{-1}\circ e_{-1}=a_0(e_{-1},e_{-1})e_0.\]
This defines clearly a Poisson structure on $\G_\la$
and it is the only one.

Let us show now that $\G_\la$ endowed with the product $u.v=[u,v]+u\circ v$ is a symmetric Leibniz algebra and $\textbf{k}_\la(u.v,w)=\textbf{k}_\la(u,v.w)$. The last assertion follows from the bi-invariance of $\textbf{k}_\la$ and one can check easily that the conditions of Proposition 2.11 in \cite{be} are trivially satisfied.
\end{proof}

\section{ Symmetric Leibniz bialgebra structures on oscillator Lie algebras}
\label{section4}

In this section, we will determine the symmetric Leibniz bialgebra structures on the generic  oscillator Lie algebras.
We apply Theorem 1.1 in \cite{bm} to provide the oscillator Lie algebra $\G_\la$ with Lie bialgebra structures and then the Theorem 3.8 in \cite{be} to endow $\G_\la$ with symmetric Leibniz bialgebra structures.

A main goal of our work is to present  symmetric Leibniz bialgebra structures on the class of generic  oscillator Lie algebras.
\begin{theo}\label{bel} Let $(\G_\la,[\;,\;])$ be a generic oscillator Lie algebra, $"."$ the   symmetric Leibniz product on $\G_\la$ given in Theorem \ref{poisson} and $\De:\G_\la\too \G_\la\otimes\G_\la$ a coproduct. Then $(\G_\la,.,\De)$ is   a symmetric Leibniz bialgebra  if and only if there exists a nonzero $\ga\in\R$, $u_0\in S$, $r\in\wedge^2S$ and $\mu\in\R^n$
\begin{eqnarray*}
&& \De(e_0)=0, \\
&& \De(e_{-1})=\ga e_0\odot e_0+ \ad_{e_{-1}} r-2e_0\wedge D(u_0),\\
&& \De(u)=\ad_u r+2e_0\wedge J^\mu(u),  \quad \mbox{ for } u\in S
\end{eqnarray*}
  and $r$ satisfies
	\begin{equation}\label{bouc}\om_{r,\ad_{e_{-1}} r}-(J^\mu\circ \ad_{e_{-1}})  r=0.\end{equation}
\end{theo}
\begin{proof} According to Theorem \ref{be},  $\De=\De_{\mathfrak{L}}+\De_a$ with $(\G_\la,\br,\De_{\mathfrak{L}})$ is a Lie bialgebra and $\De_a$ satisfies the conditions 2-5 of Theorem \ref{be}. According to Theorem \ref{bm}, for any $u\in\G_\la$,
\[ \De_{\mathfrak{L}}(u)=\ad_u r+2e_0\wedge ((J^\mu+\ad_{u_0})(u)), \]
$\mu\in\R^n$, $u_0\in S$ and  $r\in\wedge^2S$   satisfying \eqref{bou}.
Thus
\begin{eqnarray*}
&& \De_{\mathfrak{L}}(e_0)=0, \\
&& \De_{\mathfrak{L}}(e_{-1})=\ad_{e_{-1}} r-2e_0\wedge Du_0,\\
&& \De_{\mathfrak{L}}(u)=\ad_u r+2e_0\wedge J^\mu(u), \quad \mbox{ for } u\in S.
\end{eqnarray*}

On the other hand,
 Condition 3. in Theorem \ref{be} implies that
\[ \De_a(e_0)=\De_a(u)=0,\quad u\in S. \] To check the other conditions in Theorem \ref{be} we write
\[ \De_a(e_{-1})=\al e_{-1}\odot e_{-1}+\be e_{-1}\odot e_{0}+\ga e_{0}\odot e_{0}+ e_{-1}\odot v_{-1}+ e_{0}\odot v_{0}+s, \]
with $\al , \be , \ga \in \R$, $v_0,v_{-1}\in S$ and $s\in \bigodot^2S$.

Note first that if $\circ=\De_a^*$ the dual associative commutative product on $\G_\la$ then, for any $\al,\be\in\G_\la^*$,
\[ \al\circ\be= \De_a(e_{-1})(\al,\be) e_{-1}^* \]and $(\G_\la^*,\De_{\mathcal L},\circ)$ is a symmetric Leibniz algebra.  According to \cite[Proposition 2.11]{bb}, $\al\circ\be$ is in the annulator of $\circ$ and hence $e_{-1}^*$ is in the annulator. This implies
\[ \De_a(e_{-1})=\ga e_{0}\odot e_{0}+ e_{0}\odot v_{0}+s.  \]
Since for the  commutative associative product on $\G_\la$ the only nonzero product is $e_{-1}\bullet e_{-1}=ce_0$, the conditions 2-6 in Theorem \ref{be} are equivalent to
\[ (\De_{\frak L} \otimes I_{\frak L}) \circ \De_a = 0\esp \De_a(e_{-1})\cdot_{\frak L}x=x\cdot_{\frak L}\De_a(e_{-1})=0, \] for any $x\in\G_\la$.

Let us write $ s = \sum_{i=1}^m p_i \otimes q_i $,  with $p_i ,q_i \in S$ when $   \: { i \in \{1,\ldots ,m\}}$, such that $\{p_1,\ldots ,p_m\}$ is a set of linear independent vectors in $S$. Since the center of $\G_\la$ is $ \R e_0$ then we have, for $u\in S$,
\begin{eqnarray*}
\begin{split}
\De_a(e_{-1})._{\mathfrak{L}}u
& =  \om(v_0,u)e_0\otimes e_0+\sum_{i=1}^m p_i \otimes [q_i,u]\\
&=\om(v_0,u)e_0\otimes e_0+\left( \sum_{i=1}^m\om(q_i,u)p_i\right)\otimes e_0.
\end{split}
\end{eqnarray*}The vanishing of this bilinear form is equivalent to
\[ \om(v_0,u)=\om(q_i,u)=0,\quad i=1,\ldots,m \]for any $u\in S$. But $\om$ is nondegenerate and hence $v_0=0$ and $s=0$. This completes the proof.
\end{proof}

\section{Geometric consequences}\label{section5}

Let $G$ be a connected Lie group and $(\G,\br)$ its Lie algebra. It has been shown in \cite{bb} that there is  a bijection
between the set of Poisson structures on   $(\G,\br)$ and the space of
bi-invariant torsion free linear connections  on  $G$ which have the same curvature
as $\na^0$ where $\na^0$ is the bi-invariant connection on $G$ given by $\na^0_XY=\frac12[X,Y]$. Moreover, if the Poisson structure is symmetric Leibniz the corresponding bi-invariant connection is locally symmetric and has the same holonomy algebra as $\na^0$. In this section, we give the expression in the canonical coordinates of the bi-invariant connection on the oscillator Lie group associated to the Poisson structure defined in Theorem \ref{poisson}. Actually, we give also the expression of the left invariant Lorentzian metric associated to $\textbf{k}_\la$ and its Levi-Civita connection $\na^0$.

For $n \in \N^*$ and $\la = (\la_1,\ldots,\la_n) \in \R^n$ with $0 < \la_1 \leq \cdots \leq \la_n$, the $\la$-oscillator group, denoted by $G_\la$, is the Lie group with the underlying manifold $\R^{2n+2} = \R \times \R\times \C^n$ and product $$(t,s,z).(t',s',z') = \Bigl( t+t',s+s'+\frac12 \sum_{j=1}^n \mbox{Im }[\bar{z}_j \exp(it\la_j)z_j'],
\ldots, z_j+\exp(it\la_j)z'_j, \ldots \Bigr) .$$
It is Lie algebra is $\G_\la$ is $\R \times \R\times \C^n$ with its canonical basis
$\B = \left\{e_{-1},e_0,e_j,\ch_j\right\}_{j=1,\ldots,n}$ with
\[ e_{-1}=(1,0,0),e_0=(0,1,0),\; e_j=(0,0,(0,\ldots,1,\ldots,0))\esp \ch_j=
(0,0,(0,\ldots,\imath,\ldots,0)). \]
For any $u\in\G_\la$, we denote by $u^\ell$ the left invariant field on $G_\la$ associated to $u$. We have
\[ u^\ell(t,s,z)=\frac{d}{d\mu}_{|\mu=0}(t,s,z).(\mu u). \]So if $u=(t_0,s_0,z^0)$ we get
\begin{eqnarray*}
	u^\ell&=&\frac{d}{d\mu}_{|\mu=0}
	\Bigl( t+\mu t_0,s+\mu s_0+\frac12 \sum_{j=1}^n \mbox{Im }(\bar{z}_j \mu\exp(it\la_j)z^0_j),
	\ldots, z_j+\mu\exp(it\la_j)z^0_j, \ldots \Bigr) \\
	&=&\Bigl(  t_0, s_0+\frac12 \sum_{j=1}^n \mbox{Im }(\bar{z}_j \exp(it\la_j)z^0_j),
	\ldots, \exp(it\la_j)z^0_j, \ldots \Bigr).
\end{eqnarray*}In the linear coordinates $(t,s,x_j,y_j)_{j=1,\ldots,n}$ associated to $\B$ we get
\begin{eqnarray*}
	u^\ell&=&t_0\partial_{t}+\left[s_0+\frac12 \sum_{j=1}^n \mbox{Im }(\bar{z}_j \exp(it\la_j)z^0_j)\right]\partial_{s}\\&&+
	\sum_{j=1}^n\left[ (x^0_j\cos(t\la_j)-y^0_j\sin(t\la_j))\partial_{x_j}
	+(x^0_j\sin(t\la_j)+y^0_j\cos(t\la_j))\partial_{y_j} \right],\end{eqnarray*}where
$u=(t_0,s_0,x_j^0,y_j^0)$. We deduce that
\begin{equation}\label{eq}\left\{\begin{array}{l}\di
	e_{-1}^\ell=\partial_{t},\;e_0^\ell=\partial_{s},\\\di
	e_j^\ell=\frac12(x_j\sin(t\la_j)-y_j\cos(t\la_j))\partial_{s}+
	\cos(t\la_j)\partial_{x_j}
	+\sin(t\la_j)\partial_{y_j},\\\di
	\ch_j^\ell=\frac12(x_j\cos(t\la_j)+y_j\sin(t\la_j))\partial_{s}-
	\sin(t\la_j)\partial_{x_j}
	+\cos(t\la_j)\partial_{y_j}.\end{array}\right.
\end{equation}From this relations we deduce that
\begin{equation}\label{eq2}\left\{\begin{array}{l}\di
	\partial_t=e_{-1}^\ell,\;\partial_s=e_0^\ell,\\\di
	\partial_{x_j}=\frac12y_je_0^\ell+ \cos(t\la_j)e_j^\ell-\sin(t\la_j)\ch_j^\ell,\\\di
	\partial_{y_j}=-\frac12x_je_0^\ell+ \sin(t\la_j)e_j^\ell+\cos(t\la_j)\ch_j^\ell.
\end{array}\right.
\end{equation}

Let $h_\la$ denote the Lorentzian left invariant metric on $G_\la$ associated to $\textbf{k}_\la$. We have
\[ h_\la(e_{-1}^\ell,e_0^\ell)=1\esp h_\la(e_i^\ell,e_i^\ell)=
h_\la(\ch_i^\ell,\ch_i^\ell)=\frac1{\la_i}\]and all the other products are zero. Thus
\[ h_\la=2dtds+\left[\sum_{i=1}^n(y_idx_i-x_idy_i)\right]dt+\sum_{i=1}^n\frac1{\la_i}(dx_i^2+dy_i^2). \]The Levi-Civita $\na^0$ of $h_\la$ is given by
\begin{equation}\label{lv}\na_{u^\ell}^0 v^\ell=\frac12[u^\ell,v^\ell],\quad u,v\in\G_\la.\end{equation} A direct computation using \eqref{eq}-\eqref{lv} shows that expression of $\na^0$ in the coordinates $(t,s,x_i,y_i)$ is quite simple and given by
\begin{equation}\label{lv1}
	\na_{\partial_{t}}^0\partial_{x_j}=\na_{\partial_{x_j}}^0\partial_{t}=-\frac{\la_j}{2}\left(\frac12 x_j\partial_s+\partial_{y_j}\right)\esp
	\na_{\partial_{t}}^0\partial_{y_j}=\na_{\partial_{y_j}}^0\partial_{t}=-\frac{\la_j}{2}\left(\frac12 y_j\partial_s-\partial_{x_j}\right),\;
\end{equation}$j=1,\ldots,n,$ and all the other Christofell symbols vanishes.

Let $\circ$ the Poisson product on $\G_\la$ defined in Theorem \ref{poisson}. By virtue of \cite[Theorems 2.1 and 3.1]{bb}, the connection $\na$ on $G_\la$ given by
\[ \na_{u^\ell} v^\ell=\na^0_{u^\ell}v^\ell+(u\circ v)^\ell \]is bi-invariant, locally symmetric and has the same curvature and holonomy as $\na^0$. According to \eqref{lv1} and the expression of $\circ$, $\na$ is given in the coordinates $(t,s,x_i,y_i)$ by
\begin{eqnarray*}\label{lv2}
\na_{\partial_{t}}\partial_{x_j}&=&\na_{\partial_{x_j}}\partial_{t}=-\frac{\la_j}{2}\left(\frac12 x_j\partial_s+\partial_{y_j}\right),\\
\na_{\partial_{t}}\partial_{y_j}&=&\na_{\partial_{y_j}}\partial_{t}=-\frac{\la_j}{2}\left(\frac12 y_j\partial_s-\partial_{x_j}\right),\\
\na_{\partial_{t}}\partial_{t}&=&c\partial_s,
\end{eqnarray*}$j=1,\ldots,n,$ and all the other Christofell symbols vanishes.

\end{document}